\documentclass[10pt]{amsart}
\usepackage{amsmath,amsthm,amssymb}
\usepackage{graphicx}

\theoremstyle{definition}

\newtheorem{theorem}{Theorem}

\newtheorem{criteria}{Criterion}

\newtheorem{definition}{Definition}
\newtheorem{corollary}{Corollary}

\newtheorem{remark}{Remark}

\newcommand{\R}{{\mathbb R}}
\newcommand{\gT}{{\mathfrak T}}
\newcommand{\const}{{\rm const}}
\newcommand{\essup}{{\rm essup}}

\begin{document}

\title[Disconjugacy for a second order linear differential equation]{The theory of disconjugacy for a second order linear differential equation}

\author{V.~Y{a}.~Derr}

\address{Faculty of Mathematics, Udmurtia State University \\ Universitetskya St., 1 (building 4), Izhevsk, 426034, Russia}

\email{vandv@udm.net}


\begin{abstract}
This is an introduction to the theory of disconjugacy for a second order linear differential equation.
We give new proofs of some of basic results and obtain new sufficient conditions for disconjugacy (in particular, on the whole real axis).
\end{abstract}

\maketitle

The differential equation 
\begin{equation}
\label{eq1}
(Lx)(t):= x''+p(t)x'+q(t)x=0,
\end{equation}
is called {\it disconjugate} on an open interval $J \subset \mathbb R$ if any of its non-trivial solutions have at most one zero in $J$. 
The property of disconjugacy, which guarantees the existence of the unique 
solution to the boundary value problem
\begin{equation*}
x''+p(t)x'+q(t)x=f(t)\quad ,\quad x(a)=0,\;x(b)=0,
\end{equation*}
was discovered in 1951 by A.~Wintner \cite{wint51}, and since 
then attracted the interest of many mathematicians \cite{AzTz}--\cite{muld78}, 
in particular due to its great importance to qualitative theory of differential equation (\ref{eq1}).

In this article we describe the present state of the theory of disconjugacy for differential equation (\ref{eq1}) (Sections 1--8).

We also obtain a new sufficient condition for disconjugacy of (\ref{eq1}) (Sections 9 and 10).
Traditionally (see, e.g., \cite{pol24} --- \cite{muld78}), the conditions for disconjugacy are obtained for differential equations of the form
\begin{equation}
\label{eqQ}
x''+Q(t)x=0
\end{equation}
and include the assumptions of smallness of coefficient $Q$.
Our condition does not necessarily require the smallness of $q$.

\vspace*{3mm}

\textbf{1. }We start with recalling some definitions. Let us consider differential equation 
\begin{equation}
\label{eq2}
(Lx)(t)=f(t)
\end{equation}
where
\begin{equation*}
t\in I:=(a,b),\quad -\infty\leqslant a<b\leqslant +\infty, \quad p,~q,~f:I \mapsto \mathbb R \text{ are locally summable}.
\end{equation*}
A function $x:I\to \R$ is called \textit{solution} of equation (\ref{eq2}) if it has locally absolutely continuous first derivative $x'$ and satisfies equation (\ref{eq2}) almost everywhere 
(with respect to Lebesgue measure). Under our assumptions on $p,q$ there exists the unique solution of equation (\ref{eq2}) satisfying $x(a)=\xi_0,\;x'(a)=\xi_1,\;a\in I.$ 
Recall that the general solution of (\ref{eq1}) has form $x(t)=c_1u_1(t)+c_2u_2(t),$ where $u_1,u_2$ are linearly independent solutions of (\ref{eq1}), 
$c_1,c_2$ are arbitrary constants; the pair $\{u_1,u_2\}$ is called \textit{fundamental system} of (\ref{eq1}); its
Wronskian $W(t):=[u_1,u_2](t)$ is nowhere zero on $I.$ 

Function $C:I\times [\alpha,t]\to \R$ is called \textit{Cauchy's function} of equation (\ref{eq1}) if
$$
(LC)(\cdot ,s)=0\quad \text{for almost all}\;t\geqslant s, \quad C(s,s)=0,\quad \frac{\partial C(s,s)}{\partial t}=1\quad (s\in I).
$$
We note that Cauchy's function always exists and is unique. One can represent the general solution of equation (\ref{eq2}) in the form
\begin{equation}\label{eq3}
x(t)=c_1u_1(t)+c_2u_2(t)+\int_a^t C(t,s)f(s)\,ds
\end{equation}
where $(u_1,u_2)$\, is the fundamental system of solutions of (\ref{eq1}) and $c_1,c_2$ are arbitrary constants.

Function $G:[a,b]^2\to \R$ is called \textit{Green's function} of boundary value problem
\begin{equation}\label{eq4}
(Lx)(t)=f(t)\quad (t\in I),\quad x(a)=0,\;x(b)=0\quad (a,b\in I),
\end{equation}
provided that it satisfies the following conditions:

1) $G$ is continuous on $[a,b]^2;$

2) $\frac{\partial G(\cdot,s)}{\partial t}$ is absolutely continuous in the triangles $a\leqslant s<t\leqslant b$ and $a\leqslant t<s\leqslant b,$ and
$$\frac{\partial G(s+,s)}{\partial t}-\frac{\partial G(s-,s)}{\partial t}=1;$$

3)$(LG)(\cdot ,s)=0\;\text{при}\;t\ne s;$

4) $G(a,s)=0,\;G(b,s)=0.$

If boundary value problem (\ref{eq4}) has the unique solution, then it has the unique Green's function, and its solution $x$ admits presentation
$$x(t)=\displaystyle{\int_a^b} G(t,s)f(s)\,ds.$$ Also, one has the following identity
$$
G(t,s)= 
\begin {cases} 
-\frac {C(b,t)C(s,a)}{C(b,a)},&\text{ if }\; a \leqslant s < t ,\\ 
-\frac {C(t,a)C(b,s)}{C(b,a)},&\text{ if }\; t\leqslant s \leqslant b,
\end {cases}
$$
which implies that if $C(t,s)>0,$ for $a\leqslant s<t\leqslant b,$ then $G(t,s)<0$ for $(t,s)\in (a,b)^2.$

\vspace*{2mm}

\textbf{2. }We will also need the following two results due to Sturm: Separation of Zeros Theorem and Comparison Theorem \cite[p. 252]{stepvv},\,\cite[p. 81]{hart70}.

\begin{theorem}[Separation of Zeros]
\label{thsep} 
Let $a$, $b\in I$, suppose that $x$ is a solution of equation {\rm (\ref{eq1})} such that $x(a)=x(b)=0,$   $x(t)\ne 0$ for any $t\in (a,b).$ Then any other solution of (\ref{eq1}), linearly independent with $x$, has the only zero in  $(a,b).$
\end{theorem}

\begin{proof}
Suppose that $y$ is a solution of equation (\ref{eq1}) linearly independent with $x$ and such that $y(t)\ne 0$ on $(a,b).$ Since $y(a)\ne 0,\;y(b)\ne 0$ (due to linear independence of $x$ and $y$),
$y(t)\ne 0$ on $[a,b].$ Therefore, function $$h(t):= -\dfrac{W(t)}{y(t)},$$ where Wronskian $W$ of $\{x,y\}$ is continuous and nowhere zero on $[a,b]$, hence $h(t)\ne 0$ on $[a,b].$  
Without loss of generality $h(t)>0$ on $[a,b].$ Since $$h=\left(\dfrac{x}{y}\right)',$$ we have
$\int_a^bh(t)\,dt>0$ and, at the same time, $$\int_a^bh(t)\,dt=\dfrac{x(b)}{y(b)}-\dfrac{x(a)}{y(a)}=0.$$ The latter implies that $y(t^*)=0$ at some $t^*\in (a,b).$

If $y(t_*)=0$ at some $t_*\ne t^*,$ then, as we already proved, $x$ would have a zero in $(a,b),$ which contradicts to our assumptions.
\end{proof}

Let $a\in I,$ $x$ be a solution of equation (\ref{eq1}) such that $x(a)=0$. Point $\rho_+(a)>a$ $\bigl(\rho_-(a)<a\bigr)$ is called {\it right} ({\it left}) 
{\it conjugate point of $a$}  if
$$
x(\rho_{\pm}(a))=0,\; x(t)\ne 0\quad \text {in}\quad (a,\rho_+(a))\quad  \bigl(\text{in}\;(\rho_-(a),a)\bigr).
$$
If $x(t)\ne 0$ on $(a,\beta)$ $\bigl(\text{respectively}, (\alpha , a)\bigr),$ we define $\rho_+(a)=\beta\quad \bigl(\rho_-(a)=\alpha\bigr)$

\begin{corollary}\label{incrrho}
{\it Functions  $\rho_{\pm}$\, are strictly increasing. 
Furthermore, $\rho_+\bigl(\rho_-(t)\bigl)=\rho_-\bigl(\rho_+(t)\bigl)=t\quad (t\in I),$ i.e., functions $\rho_{\pm}$ are the inverses of each other and map continuously any interval in $I$ to an interval in $I.$}
\end{corollary} 


\begin{proof}
Let $t_2>t_1,\;x(t_1)=y(t_2)=0$\; ($x$ and $y$\, are solutions of {\rm (\ref{eq1})}). Suppose that $\rho_+(t_2)\leqslant \rho_+(t_1).$  The equality here, meaning that
$x(\rho_+(t_2))=y(\rho_+(t_1))=0$, contradicts to the definition of a conjugate point. Meanwhile, the strict inequality contradicts to Theorem \ref{thsep} (since $y$ would have two zeros between two consequtive zeros of $x.$) Consequently, $\rho_+(t_2)> \rho_+(t_1).$  

The proof for $ \rho_-$ is similar. The proof of the second statement follows from the definition of conjugate points and properties of strictly monotone functions.
\end{proof}

\begin{definition}
Differential equation (\ref{eq1}) is called {\it disconjugate} on an interval $J\subset I$ if any of its non-trivial solutions has at most one zero in $J$.
(We also say that $J$ is an interval of disconjugacy of equation  (\ref{eq1})).
\end{definition}

Thus, $J$ is an interval of disconjugacy of equation
 (\ref{eq1}) if and only if $\rho_{\pm}(a)\notin J$ for any $a\in J.$ 

In what follows, we write $L\in \gT(J)$ if equation (\ref{eq1}) is disconjugate on interval $J\subset I.$ 


Let $(a,b)\subset I,$ suppose that $a_n\to a+,\;b_n\to b-$ ($a_n\to -\infty,\;b_n\to +\infty$ in the case $a=\alpha=-\infty,\;b=\beta =+\infty$). Then
\begin{equation}\label{eqdop}
\gT\bigl((a,b)\bigr)=\bigcap\limits_{n=1}^{\infty}\gT\left(\bigl[a_n,\,b_n\bigr]\right)=\bigcap\limits_{n=1}^{\infty}\gT\left(\bigl(a_n,\,b_n\bigr)\right).
\end{equation}

As follows from the definitions of the property of disconjugacy and Cauchy's function, if 
equation (\ref{eq1}) is disconjugate on the interval $J=[a,b)\subset I,$ then $C(t,s)>0$ in the triangle
$a\leqslant s<t<b.$ Disconjugacy of equation (\ref{eq1}) on an interval $[a,b]$ implies the existence of the unique solution of problem (\ref{eq4}), so Green's function of this problem satisfies $G(t,s)<0$ on $(a,b)^2.$

\begin{theorem}[Comparison Theorem]
\label{thcomp} 
Let
$$
L_i\,y:= y''+p(t)\,y'+q_i(t)\,y=0,\quad i=1,2\quad\text{и}\quad q_1(t)\leqslant  q_2(t)\quad (t\in I).
$$ 
If $L_2\in \gT(J),$ then
$L_1\in \gT(J)$.
\end{theorem}

\begin{proof}
Let $L_2\in \gT(J).$ Suppose that $L_1\notin \gT(J).$  Then there exist a solution $x$ of equation $L_1x=0$, and points $a,b\in J,\;a<b,\;x(a)=x(b)=0,\;x(t)>0\;(t\in (a,b)).$
Let us define the solution $y$ of equation $L_2y=0$ by initial values $y(a)=0,\;y'(a)=x'(a)\,(>0).$ Since $L_2\in \gT(J),$ then $y(b)>0.$ Let $h:= y-x.$ 
Then $0=L_2y-L_1x=L_2h-(q_1-q_2)x, \;h(a)=0,h'(a)=0.$ Also, note that $h(b)>0.$ Thus, function $h$ is a solution to problem
$$
(L_2h)(t)=(q_1(t)-q_2(t))x(t),\quad h(a)=h'(a)=0,
$$
i.e.,
$$
h(t)=\int_a^tC_2(t,s)\bigl(q_1(s)-q_2(s)\bigr)x(s)\,ds\leqslant 0, 
$$ 
since Cauchy's functions $C_2(t,s)$ of equation $L_2y=0$ is positive for all $a\leqslant s<t\leqslant b,\; x(s)>0,$\\ $\bigl(q_1(s)-q_2(s)\bigr)\leqslant 0$ for all $a<s \leqslant t\leqslant b.$
The inequality $h(b)\leqslant 0$ contradicts to the inequality $h(b)>0$ obtained previously, so $L_1\in \gT(J).$
\end{proof}

\textbf{3.}
We use Theorem \ref{thsep} and Corollary \ref{incrrho} to prove the following statement.
\begin{theorem}\label{th21} 
Let $a$, $b \in I.$ Then 
$$
\gT\bigl((a,b)\bigr) = \gT\bigl((a,b]\bigr) = \gT\bigl([a, b)\bigr).
$$
\end{theorem}

\begin{proof} 
It suffices to check inclusion $\gT\bigl((a,b)\bigr) \subset \gT\bigl([a,b)\bigr),$ the rest follows by symmetry.

Let
$L\in \gT\bigl((a,b)\bigr).$ Suppose that $L\notin \gT\bigl([a,b)\bigr).$ Then there exist a solution $v$ of equation (\ref{eq1}) such that $v(a)=v(c)=0\; (a<c<b)$, $v(t)>0$ in $(a,c)$ (note that $v$ 
has at least two zeros in $[a,b),$ and at most one zero in $(a,b)$). By definition, $c=\rho_+(a)\quad \bigl((a=\rho_-(c)\bigr).$ Let us choose  $c_1 \in (c,b)$ so that $v(t)<0$ in $(c,c_1)$. We put $a_1=\rho _{-}(c_1).$ According to Corollary \ref{incrrho} $a<a_1<c.$ 
Let $x$ be the corresponding solution of equation (\ref{eq1}), i.e., $x(c_1)=x(a_1)=0,\quad x(t)>0$ in $(a_1,c_1).$
Since $L\in \gT\bigl([a_1,c_1]\bigr),$ there exists the unique solution $y$ of equation (\ref{eq1}) that satisfies $y(a_1)=y(c_1)=1.$ We have $y(t)>0$ on $[a_1,c_1]$  
(as a continuous function taking the same values at the endpoints of the interval, function $y$ can have only even number of zeros, hence, due to disconjugacy, none of them). We note that $y$ is linearly independent with $v$ and with $x.$ According to Theorem \ref{thsep} $y$ has exactly one zero in both intervals $(a,a_1)$ and $(c,c_1),$ that is, $y$ has two zeros in $(a,b).$ The latter contradicts to disconjugacy of equation (\ref{eq1}) on $(a,b)$.
\end{proof}

\begin{theorem}\label{thnefkrit}
$1.$ If there exists a solution of equation (\ref{eq1}) that is nowhere zero on $[a,b]\subset  I \\ \bigl((a,b)\subset I\bigr),$ 
then $L\in \gT\bigl([a,b]\bigr)\quad \bigl(L\in\gT\bigl((a,b)\bigr)\bigr).$

$2.$ If $L\in \gT([a,b],\;[a,b]\subset  I)\quad \bigl(L\in \gT([a,b)),\;[a,b)\subset I \bigr),$ then there exists a solution of equation (\ref{eq1}) that is nowhere zero on $[a,b]\quad 
\bigl((a,b)\bigr).$
\end{theorem}

\begin{proof}
1. The statement follows immediately from Theorem \ref{thsep}.

2. Let $J=[a,b],\;[a,b]\subset  I.$ Let us determine solutions $y_1(t)$ and $y_2(t)$ by initial conditions $y_1(a)=0,\, y_1'(a)=1$ and
$y_2(b)=$~$0,$ $y_2'(b)=-1.$ 
Since $L\in\gT([a, b]),$ one has 
$$
y_1(t)>0\quad (t\in (a, b]), \quad y_2(t)>0\quad (t\in [a, b)).
$$
The solution $y_1(t) + y_2(t)$ is the one required.

If $L\in \gT([a,b)),$ then the required solution is $y_1.$ 
\end{proof}

It is possible that there are no solutions preserving sign on $[a,b)$. For instance,
$${L:=\frac{d^2}{dt^2}+1}\in\gT([0, \pi)).$$ However, any solution of equation $Lx=0$ has precisely one zero in $[0, \pi)$.

\textbf{4. }
Below we prove two theorems which demonstrate the role of disconjugacy in the theory of differential equation (\ref{eq1}). This is Factorization Theorem (i.e., the theorem on representation of $L$ as the product of linear differential operators of the first order \cite{pol24},\cite{mam31}) and generalized Rolle's Theorem (\cite[p. 63]{ps782}).

\begin{theorem}[Factorization Theorem]\label{thPM} 
Suppose $J=[a,b]\subset I$ or $J=(a,b)\subset I.$ One has $L\in \gT(J),$ if and only if there exist functions $h_i, i=0,1,2$ such that $h_0',h_1$ are absolutely continuous,
$h_2$ is summable on $J,\;h_i(t)>0,\;h_0(t)h_1(t)h_2(t)\equiv 1$ on $J$, and
\begin{equation}
\label{PMeq1}
(Lx)(t)=h_2(t)\frac{d}{dt}h_1(t)\frac{d}{dt}h_0(t)x(t)\quad (t\in J,\quad x'\;\text{absolutely continuous on}\; J).
\end{equation}
\end{theorem}

\begin{proof} 
Necessity. Let $L\in \gT(J).$ According to Theorem \ref{thnefkrit} there exists a solution $y$ of equation (\ref{eq1}) such that $y(t)>o$ on $J.$ Let $u$ be a solution of equation (\ref{eq1})
linearly independent with $y$ and such that $w(t):= [y,u](t)>0.$ Let us consider the following linear differential operator of the second order
$$
\widehat Lx:= \frac{w}{y}\frac{dt}{dt}\frac{y^2}{w}\frac{dt}{dt}\frac{x}{y}.
$$
Since functions $y,u$ form a fundamental system of solutions of both equation (\ref{eq1}) and equation $\widehat Lx=0,$  the top coefficient in $\widehat L$ is
equal to $\frac{w}{y}\frac{y^2}{w}\frac{1}{y}\equiv 1,$ then $Lx\equiv \widehat Lx.$ These conditions are satisfied if $h_0=\frac{1}{y},\;h_1=\frac{y^2}{w},\;h_2=\frac{w}{y}.$

Sufficiency. Suppose that we have identity (\ref{PMeq1}). Then function $y(t):=\frac{1}{h_0(t)}>0\;(t\in J)$ is a solution of equation (\ref{eq1}) satisfying conditions of Theorem \ref{thnefkrit} which, in turn, implies that $L\in \gT(J).$
\end{proof}

\begin{theorem}[Generalized Rolle's Theorem]
\label{thRol} 
Let $J=[a,b]\subset I$ or $J=(a,b)\subset I,\quad L\in \gT(J).$ Suppose that function $u$ has absolutely continuous on $J$ first derivative, and function $Lu$ is continuous. If there exist $m$ ($m\geqslant 2$) geometrically distinct zeros of $u$ in $J$, then $Lu$ has at least $m-2$ geometrically distinct zeros in $J$.
\end{theorem}

\begin{proof} 
According to Theorem \ref{thPM} one has representation (\ref{PMeq1}). Since function $h_0u$ has $m$ geometrically distinct zeros in $J$, by Rolle's Theorem both $\frac{d}{dt}h_0u$ and $h_1\frac{d}{dt}h_0u$ have at least $m-1$ geometrically distinct zeros in $J$. Also according to Rolle's Theorem $Lu$ has at least $m-2$ geometrically distinct zeros in $J$.
\end{proof}

\textbf{5.}
In this section we provide some criteria for disconjugacy based on Theorems \ref{thcomp} and \ref{thnefkrit}.

\begin{criteria}\label{const}
Let $I=(-\infty,\,+\infty)$, $p(t)\equiv p=\const$, $q(t)\equiv q=\const$. Then differential equation $(\ref{eq1})$ having constant coefficients $p(t) \equiv p$, $q(t) \equiv q$ is disconjugate on $I$ if and only if the roots of its characteristic equation
 $\lambda ^2+p\lambda +q=0$ are real.
\end{criteria}

\begin{proof}
Let $\nu$ be a real root of the characteristic equation. Then function $x(t):= e^{\nu t}$ is a solution of equation (\ref{eq1})
nowhere vanishing on $I.$ 
According to the first statement of Theorem \ref{thnefkrit}, equation (\ref{eq1}) is disconjugate on $I.$

Conversely, let (\ref{eq1}) be disconjugate on $I.$ Suppose that the characteristic equation has roots $\gamma \pm\delta i, \delta\ne 0.$ Then solution
$x(t)=e^{\gamma t}\cos\,\delta t$ of equation (\ref{eq1}) has infinitely many zeros in $I$, which contradicts to its disconjugacy on $I.$
\end{proof}

Let us consider equation
\begin{equation}\label{euler1}
x''+\frac{p}{t}x'+q(t)x=0\quad (t\in I(0,+\infty),\quad\text{where}\;p=\const .
\end{equation}

\begin{criteria}\label{euler}
If $q(t)\leqslant \frac{(p-1)^2}{4t^2},$ then equation $(\ref{euler1})$ is disconjugate on $I:= (0,+\infty).$
\end{criteria}

\begin{proof}
Euler equation $x''+\frac{p}{t}x'+\frac{(p-1)^2}{4t^2}x=0$ is disconjugate on $I$ by Theorem \ref{thnefkrit} since it has solution $x(t)=t^{\frac{1-p}{2}},$ 
which is nowhere equal to zero on $I$ (let us also take into account (\ref{eqdop})). 
According to Theorem \ref{thcomp} equation (\ref{euler1}) is also disconjugate on this interval.
\end{proof}

The next sufficient condition of disconjugacy is due to A.M.~Lyapunov
\cite{cop71}.

\begin{criteria}\label{ljap}
Let $p(t)\equiv 0, \quad q(t) \geqslant 0$ and $\int_a^b q(t)dt \leqslant \frac {4}{b-a}.$ Then $L \in \gT([a, b]).$
\end{criteria}

\begin{proof}
Suppose that equation (\ref{eq1}) possesses a non-trivial solution $y(t)$ having two zeros in $[a,b].$
Since $y$ can not have multiple roots, we may assume, without loss of generality, that
\begin{equation}\label{ljap1}
y(a) = y(b) = 0. 
\end{equation}
Function $y$, as a solution of boundary value problem (\ref{eq1}), (\ref{ljap1}), satisfies the following integral equation 
\begin{equation}\label{ljap2}
y(t)= -\int_a^b G(t,s)q(s)y(s) ds, 
\end{equation}
where
$$
G(t,s)= 
\begin {cases} 
-\dfrac {(b-t)(s-a)}{b-a},&\text{if}\; a \leqslant s < t ,\\ 
-\dfrac {(t-a)(b-s)}{b-a},&\text{if}\; t\leqslant s \leqslant b 
\end {cases}
$$
is Green's function of equation $y''=0$ with boundary conditions (\ref{ljap1}). It is immediate that for $t \ne s$ 
\begin{equation}\label{ljap3}
|G(t, s)| < \frac {(b-s)(s-a)}{b-a}.
\end{equation}
Let $\max\limits_{s \in [a, b]}|y(s)| = |y(t^*)|$. Then (\ref{ljap2}) and (\ref{ljap3}) 
\begin {multline*}
|y(t^*)|= \left | \int_a^b G(t^*, s)q(s)y(s)ds \right | \leqslant \int_a^b \left |G(t^*, s) \right ||y(s)|q(s)ds < \\ 
<|y(t^*)|  \int_a^b \frac {(b-s)(s-a)q(s)}{b-a}ds  \leqslant \frac {b-a}{4} \int_a^b q(s)ds 
\end {multline*} 
since $(b-s)(s-a)\! \leqslant \!\dfrac {(b-a)^2}{4}$ for $s\! \in \![a, b].$ Therefore, $1\!\! < \!\!\dfrac {b-a}{4} \displaystyle{\int_a^b} q(s)ds,$ which contradicts to the conditions of the theorem.
\end{proof}

\begin{corollary}\label{slthljap}
If $p(t)\equiv 0,$
$$
\int_a^b q_{+}(t)\,dt\leqslant \frac{4}{b-a},
$$
then $L\in \gT([a,b])$ $\left( q_{+}(t)=q(t)\quad\text{if}\quad q(t)>0, \quad q_{+}(t)=0\quad\text{if}\quad q(t)\leqslant 0 \right).$
\end{corollary}

\begin{proof}
As we already proved, $L_{+}:= \frac{d^2}{dt^2}+q_{+}\in \gT([a,b])$. At the same time, since $q(t)\leqslant q_{+}(t)$, one has $L\in \gT([a,b]).$
\end{proof}

\begin{remark}\label{z1} 
We note that constant $4$ is formulation of Criterion $\ref{ljap}$ is sharp. 
\end{remark}

The latter follows from the next example.
Suppose that function $v$ is twice continuously differentiable on $[0,1]$ and
$$
v(t)=t \quad (0 \leqslant t \leqslant \frac{1}{2}-\delta),\quad
v(t)=1-t \quad \text{ if } \quad t>\frac{1}{2}+ \delta, 
$$
$$
v(t)>0, \quad v''(t)<0 \quad \text{ if } \quad \frac{1}{2}-\delta < t < \frac{1}{2}\!+\! \delta.
$$
Define
$$
q(t)= 
\begin {cases} 
-\frac {v''(t)}{v(t)},&\text{if} \; t \in (0, 1) ,\\ 
0,&\text{if} \;t = 0,  \; t = 1. 
\end {cases}
$$
Clearly, $q$ is continuous, $q(t)\! \geqslant \!0$ on $[0, 1];\; L\! := \!\frac {d^2}{dt^2} + q(t) \notin \gT([0, 1]),$ since equation 
$Ly=0$ has solution $y=v(t)$ which has two zeros in $[0,1].$ However,  
$$
\frac {v''}{v} = \left ( \frac {v'}{v} \right )' + {\left (\frac {v'}{v} \right )}^2 \geqslant \left ( \frac {v'}{v} \right )',
$$
so the value of integral
$$
\int_0^1 q(t)dt = -\int_{\frac{1}{2}-\delta}^{\frac{1}{2}+\delta} \left ( \frac {v'}{v} \right )'dt= 
\left. -\frac {v'}{v}  \right |_{\frac{1}{2}-\delta}^{\frac{1}{2}+\delta} = \frac{4}{1-2\delta}
$$ 
can be made arbitrarily close to $4$ by choosing sufficiently small $\delta.$

\vspace*{4mm}

\textbf{6.}
Criterion \ref{thnefkrit} is an example of a \textit{non-effective} criterion of disconjugacy, i.e., a criterion formulated rather in terms of solutions of equation
(\ref{eq1}) than in terms of the coefficients of this equation. 

Let us now give a necessary and sufficient condition of disconjugacy of equation
(\ref{eq1}). This criterion may be called \textit{semi-effective}
\cite{lev69} (it is effective as a necessary condition, but non-effective as a sufficient condition).
Although it is not expressed in terms of the coefficients of equation (\ref{eq1}), it can be used to obtain sufficient conditions of disconjugacy formulated in terms of the coefficients of the equation. 
The author of this criterion is Valle-Poussin \cite{vp29}.
\begin{theorem}\label{thVP}   
Let $[a, b] \subset I.$ One has $L \in \gT([a, b])$
if and only if there exists function
$v$ having first derivative absolutely continuous on $[a, b]$ and such that
\begin{equation}\label{vp1}
v(t) > 0\quad  (a < t \leqslant b), \qquad  Lv \leqslant 0 \quad \text{a.e. on}\; [a,b]. 
\end{equation} 
\end{theorem}
\begin{proof}
Necessity follows from Theorem \ref{thnefkrit}. Let us show that the conditions of the theorem are sufficient. In the case $v(a) = 0$ let us put
$\widetilde v(t) = v(t)+\varepsilon u(t)$, where $\varepsilon > 0,$ and $u(t)$ is the solution of equation (\ref{eq1}) with initial conditions
$u(a)=1, \; u'(a)=0$. For $\varepsilon$ sufficiently small we have $\widetilde v(t) > 0$ on $[a, b].$
Hence we may assume, without loss of generality, that
$v(t) >0$ on $[a, b].$  Let us consider equation
\begin{equation}\label{vp2}
Mx:= x''+px'-\frac{v''+pv'}{v}x=0.
\end{equation} 
According to Theorem \ref{thnefkrit} $M \in \gT([a, b])$ (since equation (\ref{vp2}) has solution $v$ positive on $[a, b]$). 
By our assumptions $v''(t) + p(t)v'(t)+q(t)v(t) \leqslant 0,$ i.e., $- \frac {v''(t)+p(t)v'(t)}{v(t)} \geqslant q(t),$  a.e. on $[a,b] .$  
The statement of the theorem now follows from Theorem \ref{thcomp}.
\end{proof}

The proof of the next statement follows the same argument.

\begin{theorem}\label{VP'} 
If there exists function $v$ having first derivative absolutely continuous on $[a, b)$ ans such that 
\begin{equation}\label{vp3}
v(t) > 0\quad (a < t < b ), \qquad Lv \leqslant 0 \quad \text{ a.e. on}\quad (a,b), 
\end{equation} 
then  $L \in \gT([a, b)).$
\end{theorem}

\textbf{7.}
By choosing a particular `test' function $v$ we can get various effective conditions for disconjugacy.

\begin{criteria}\label{A}
If $q(t)\leqslant 0$ on $[a,b]\subset I\quad \bigl(\text{в}\;(a,b)\subset I\bigr)$, then $L\in\gT([a,b])\quad \left(L\in \gT\bigl((a,b)\bigr)\right).$
\end{criteria}

\begin{proof}
We put $v(t)\equiv 1$ and then use Theorem \ref{thVP} \;(Theorem \ref{VP'}).
\end{proof}

\begin{criteria}\label{B}
Suppose that $p(t)=O(t-a)$ if $t\to a+,\;p(t)=O(b-t)$ if $t\to b-$ $($in particular, $p(t)\equiv 0).$ If 
\begin{equation*}
\frac{\pi}{b-a}\cot\,\frac{\pi (t-a)}{b-a}p(t)+q(t)\leqslant \frac{\pi ^2}{(b-a)^2},
\end{equation*}
then $L \in \gT\bigl([a, b)\bigr).$
\end{criteria}

\begin{proof}
Let us choose $v(t)\equiv \sin\,\frac{\pi (t-a)}{b-a} $ and then use Theorems  \ref{VP'} ad \ref{th21}.
\end{proof}

\begin{criteria}\label{C}
Suppose that we have inequality
\begin{equation}\label{C1}
|p(t)|\cdot \left | \frac{b+a}{2}-t\right |+|q(t)|\cdot \frac{(b-t)(t-a)}{2}\leqslant 1
\end{equation}
or inequality
\begin{equation}\label{C2}
\frac{b-a}{2}\underset{t\in (a,b)}{\essup}\,|p(t)|+\frac{(b-a)^2}{8}\underset{t\in (a,b)}{\essup}\,|q(t)|\leqslant1.
\end{equation}
Then $L \in \gT\bigl([a, b)\bigr).$
\end{criteria}

\begin{proof}
Indeed, we take $v(t)\equiv \frac{(b-t)(t-a)}{2}$ and then refer to Theorems \ref{th21} and \ref{VP'}.
\end{proof}

Let us note that inequality (\ref {C2}) implies inequality (\ref{C1}).
Let $P(t,\lambda):= \lambda ^2+p(t)\lambda +q(t)$ be the `characteristic' polynomial.

\begin{criteria}\label{D}
If there exists $\nu \in \mathbb R$ such that $P(t,\nu)\leqslant 0\quad (t\in (-\infty,+\infty)),$ then equation (\ref{eq1}) is disconjugate on $(-\infty,+\infty).$
\end{criteria}

\begin{proof}
One has $v(t):= e^{\nu t}>0$ and $(Lv)(t)=e^{\nu t}P(t,\nu)\leqslant 0$ on  $(-\infty,+\infty).$ 
The rest follows from Theorem \ref{thVP}.
\end{proof}

\textbf{8. }
Let us now formulate criteria that can be obtained from Theorem
 \ref{thVP}\;(Theorem  \ref{VP'}) using a `test' function depending on coefficients of equation 
(\ref{eq1}).

{\bf{$1^o.$}}
Let us consider equation
\begin{equation}\label{ord2.50}
\widetilde Lx:=  x''+Px'+Qx=0 
\end{equation}
having constant coefficients $P$ and $Q$, in assumption that it is disconjugate on $[a,b).$
Let $v$ be the solution of boundary value problem  $\widetilde Lv=-1,\quad v(a)=v(b)=0,$ let $\widetilde C(t,s)$
be Cauchy's function of equation (\ref{ord2.50}). Then
$$
\widetilde C(t,s)>0\quad (a\leqslant s<t<b)\quad \text{and}\quad v(t)=\int_a^b M(t,s)\,ds>0\quad (t\in (a,b)),
$$
where
$$
M(t,s):=\left\{
\begin{array}{rcl}
\dfrac{\widetilde C(b,t)\cdot \widetilde C(s,a)}{\widetilde C(b,a)},\quad a\leqslant s\leqslant t \leqslant b,\\
\dfrac{\widetilde C(t,a)\cdot \widetilde C(b,s)}{\widetilde C(b,a)},\quad a\leqslant t< s \leqslant b
\end{array}
>0,\quad (t,s)\in ((a,b)\times (a,b)).
\right.
$$
Since $(Lv)(t)=-1+\bigl(p(t)-P\bigr)v'(t)+\bigl(q(t)-Q\bigr)v(t),$ inequality $(Lv)(t)\leqslant 0$
is satisfied if
\begin{equation}\label{gener}
\bigl(p(t)-P\bigr)\int_a^b \frac{\partial M(t,s)}{\partial t}\,ds+\bigl(q(t)-Q\bigr)\int_a^b M(t,s)\,ds)ds\leqslant 1, \quad t\in (a,b).
\end{equation}
As a result, we get the following statement.

\begin{criteria}\label{XA1}
If $(\ref{gener})$ holds, then (\ref{eq1}) is disconjugate on $[a,b).$
\end{criteria}

The special choice of coefficients $P$ and $Q$ 
can lead to criteria for disconjugacy that are more subtle than the ones formulated above.

{\bf{$2^o.$}}
Consider the particular case $Q=0.$ We have
$$
M(t,s)=\left\{
\begin{array}{rcl}
\dfrac{(1-e^{-P(b-t)})(1-e^{-P(s-a)})}{P(1-e^{-P(b-a)})}\quad (s\leqslant t), \\
\dfrac{(1-e^{-P(t-a)})(1-e^{-P(b-s)})}{P(1-e^{-P(b-a)})}\quad (s>t).
\end{array}
\right.
$$
It is immediate that
\begin{multline*}
 v(t)=
\frac{(1-e^{-P(b-t)})(t-a-\frac{1}{P}(1-e^{-P(t-a)}))}{P(1-e^{-P(b-a)})}+\\ 
+\frac{(1-e^{-P(t-a)})(b-t-\frac{1}{P}(1-e^{-P(b-t)}))}{P(1-e^{-P(b-a)})}\leqslant 
\frac{2(\frac{b-a}{2}-\frac{1}{P}(1-e^{-P\frac{b-a}{2}}))}{P(1+e^{-P\frac{b-a}{2}}))}, \\
v'(t)=\frac{P\left((b-t)e^{-P(t-a)}-(t-a)e^{-P(b-t)}\right)+e^{-P(b-t)}-e^{-P(t-a)}}{P\left(1-e^{-P(b-a)}\right)}, \\
|v'(t)|\leqslant \frac{|P(b-a)+e^{-P(b-a)}-1|}{P(1-e^{-P(b-a)})}.
\end{multline*}
Since condition $Lv\leqslant 0$ is now equivalent to inequality 
$\bigl(p(t)-P\bigr)v'(t)+q(t)v(t)\leqslant 1,$ 
we get the following criterion.

\begin{criteria}\label{XA2}
If
$$
|p(t)\!-\!P| \frac{|P(b-a)\!+\!e^{-P(b-a)}-1|}{P(1-e^{-P(b-a)})}\!+\!|q(t)| 
\frac{2(\frac{b-a}{2}-\frac{1}{P}(1-e^{-P\frac{b-a}{2}}))}{P(1+e^{-P\frac{b-a}{2}}))}\!\leqslant \!1
$$ 
$(a<t<b),$ then equation (\ref{eq1}) is disconjugate on $[a,b).$
\end{criteria}

{\bf{$3^o.$}}
If we take, instead of auxiliary equation (\ref{ord2.50}), equation $\widetilde Lx:=  x''+p(t)x'=0$,
and take as $v$ the solution of problem
$\widetilde Lv=-1,\quad v(a)=v(b)=0$,
we obtain the following criterion.

\begin{criteria}\label{XA3}
If $
q(t)\,\int_a^b M(t,s)\,ds\leqslant 1, \quad t\in (a,b),
$
where
$$
M(t,s)=\left\{
\begin{array}{rcl}
\dfrac{\int_t^be^{-\int_t^{\sigma}p(\mu)\,d\mu}d\sigma\cdot \int_a^se^{-\int_a^{\sigma}p(\mu)\,d\mu}d\sigma} 
{\int_a^be^{-\int_a^{\sigma}p(\mu)\,d\mu}d\sigma}\quad (s\leqslant t), \\
\dfrac{\int_a^te^{-\int_a^{\sigma}p(\mu)\,d\mu}d\sigma\cdot \int_s^be^{-\int_s^{\sigma}p(\mu)\,d\mu}d\sigma} 
{\int_a^be^{-\int_a^{\sigma}p(\mu)\,d\mu}d\sigma}\quad (t<s),
\end{array}
\right.
$$
then equation (\ref{eq1}) is disconjugate on $[a,b).$
\end{criteria}

\textbf{9.} Let us now consider a second order criterion for disconjugacy on the whole real axis $\mathbb R$.
Let us consider differential equation
\begin{equation}\label{ord2.8}
\widetilde Lx:=  x''+px'+qx=0 
\end{equation}
having constant coefficients $p$ and $q$. As was shown before (see Criterion \ref{const}), disconjugacy of equation (\ref{ord2.8})
on $\mathbb R$ is equivalent to inequality $p^2-4q\geqslant 0.$

We will associate to equation (\ref{ord2.8}) the point $\widetilde{\mathcal L}=(p,q)$ in $p,q$-plane $\Pi$.
Let
$$
\mathfrak N:= \{(p,q):p^2-4q\geqslant 0\},\qquad \mathfrak O:= \mathbb R^2\setminus \mathfrak N.
$$
Then according to Criterion \ref{const} 
$$
\widetilde L\in\gT \bigl((-\infty,\,+\infty)\bigr) \Longleftrightarrow \widetilde{\mathcal L}\in\mathfrak N.
$$

Let us now consider differential equation
\begin{equation}\label{ord2.9}
Lx:=  x''+p(t)x'+q(t)x=0
\end{equation}
with coefficients continuous on $(-\infty,\,+\infty)$. 
Every equation of form (\ref{ord2.9}) gives rise to Jordan curve $G_L=\{t: (p(t),q(t))\}$ in plane $\Pi$.
More precisely, it determines the motion $DG_L$ along this curve. 

Consider the following problem: \textit{under which conditions the inclusion $G_L\subset \mathfrak N$ 
guaranteed disconjugacy of equation (\ref{ord2.9}) on $\mathbb R$?}

Below we formulate several possible (and simple) answers to this question.

{\bf{A.}} 
Let $p(t)\equiv p=\const$. Then inclusion $G_L\subset \mathfrak N$ is equivalent to inequality $q(t)\leqslant \frac{1}{4}p^2.$
Function $v(t):= e^{-\frac{p}{2}t}>0\;(t\in\mathbb R)$ satisfies inequality 
$$
(Lv)(t)=e^{-\frac{p}{2}t}\left(q(t)- \frac{1}{4}p^2\right)\leqslant 0\;(t\in\mathbb R).
$$
Hence,
{\it if  $p(t)\equiv p=\const,\; q(t)\leqslant \frac{1}{4}p^2,$ then equation} (\ref{ord2.9}) {\it is disconjugate on $\mathbb R$.}

{\bf{B.}} 
Let $G_L$ be a line or a segment, suppose $G_L\subset \mathfrak N.$  
The equation of such line has form either
$q(t)\equiv q=\const \leqslant 0$ (for any $p(t)$) or $p=p(t),\;q=-\gamma ^2+k\,p(t),$ where $|k|\leqslant \gamma\;(\gamma >0)$ (in the case
 $k=\pm \gamma$ the line touches parabola $q=\frac{1}{4}p^2$). 
In the first case disconjugacy of equation \ref{ord2.9}) on
on $\mathbb R$ follows from Theorem \ref{thcomp}. In the second case function $v(t):= e^{-kt}>0\;(t\in\mathbb R)$
satisfies inequality \quad $(Lv)(t)=e^{-kt}(k^2- \gamma ^2)\leqslant 0\;(t\in\mathbb R).$ Therefore,
{\it if $G_L$ is a line in $\Pi$ contained in $ \mathfrak N$, or a segment of such a line, then equation} 
(\ref{ord2.9}) {\it is disconjugate on $\mathbb R$.}

{\bf{C.}} 
Let $\gamma \geqslant 0$. Let us define
$$
\mathfrak M_{\pm} (\gamma)=\{(p,q): q\leqslant -\gamma ^2\pm\gamma \,p\} 
$$
Since both lines $q=-\gamma ^2\pm \gamma \,p$ touch parabola $q=\frac{1}{4}p^2,$ then $\mathfrak M_{\pm} (\gamma)\subset \mathfrak N$
for any $\gamma \geqslant 0$. 

Using the statements of sections ${\bf{A}}$ and {\bf{B}} and Theorem \ref{thcomp}, we get the following theorem.
 
\begin{theorem}
\label{thnew}
If for a certain $\gamma \geqslant 0$
$$
G_L\subset \mathfrak M_{+} (\gamma)\quad \bigl(G_L\subset \mathfrak M_{-} (\gamma)\bigr),
$$ 
then equation 
(\ref{ord2.9}) is disconjugate on $\mathbb R$.
\end{theorem}

(Put $v(t)=e^{-\gamma\,t}\quad \bigl(v(t)=e^{\gamma\,t}\bigr).$)

We note that conditions of Theorem \ref{thnew} depend only on curve $G_L$ but not on motion $DG_L$ along this curve.

{\bf{D.}} The conditions of the statements below now \textit{depend} on motion $DG_L.$ 
\begin{theorem}
\label{thnew2}
Suppose that $r:\mathbb R\to \mathbb R$ is a continuous function, $p$  is a differentiable function, and one of the following conditions is satisfied: 
\begin{equation}\label{dop21}
p'(t)\geqslant 2r(t)\quad (p'(t)\leqslant -2r(t))\;(t\in \mathbb R)  
\end{equation}
or
\begin{equation}\label{dop22}
p^2(t)-4p'(t)+r(t)\leqslant 0\quad (p^2(t)+4p'(t)+r(t)\leqslant 0)\; (t\in \mathbb R).
\end{equation}
Also, let
$q(t)\leqslant \frac{p^2(t)}{4}+r(t)\quad (t\in \mathbb R).$ 
Then equation $(\ref{ord2.9})$ is disconjugate on $\mathbb R$.
\end{theorem}

\begin{proof}
Suppose that we have the first inequality (\ref{dop21}). Under our assumptions, consider differential equation
\begin{equation}\label{dop23}
L_2x:= x''+p(t)x'+\left(\frac{p^2(t)}{4}+r(t)\right)\,x 
\end{equation}
Put
$$
v(t)=e^{-\frac{1}{2}{\displaystyle\int_0^t} p(s)\,ds};\quad \text{then}\quad v(t) >0,\quad (L_2v)(t)=\left(-\frac{1}{2}p'(t)+r(t)\right)\,e^{-\frac{1}{2}{\displaystyle\int_0^t} p(s)\,ds}\leqslant 0,\;t\in \mathbb R.
$$ 
Consequently, 
$L_2\in\gT ((-\infty,\,+\infty)).$ Disconjugacy of equation (\ref{ord2.9}) now follows from Theorem \ref{thcomp}.

If we have the second inequality (\ref{dop21}), then, making substitution $y(t)=x(-t),$ we arrive to equation 
$$
y''-p(t)y'+q(t)y=0$$ 
and expression
$$L_2y:= y''-p(t)y'+\left(\frac{p^2(t)}{4}+r(t)\right)\,y, 
$$
for which we define
$v(t)=e^{\frac{1}{2}{\displaystyle\int_0^t} p(s)\,ds}.$ The rest repeats the argument above.

Suppose that the first inequality (\ref{dop22}) is satisfied. Now we put  $$v(t)=e^{-\displaystyle\int_0^t p(s)\,ds}\;(>0).$$ It is easy to see that 
$$(L_2v)(t)=\left(p^2(t)-\frac{1}{2}p'(t)-p^2(t)+\frac{p^2(t)}{4}+r(t)\right)\,e^{-\frac{1}{2}{\displaystyle\int_0^t} p(s)\,ds}\leqslant 0,\;t\in \mathbb R.$$ 
Hence, $L_2\in\gT ((-\infty,\,+\infty))$, and we only need to use Theorem \ref{thcomp}. The argument in the case when the second inequality (\ref{dop22}) is satisfied is similar.
\end{proof}

Let us note that inequality (\ref{dop22}) is satisfied (in fact, as an equality) if, for instance,
$$
r(t)\equiv -R^2, \qquad p(t)=\dfrac{R\left(1-c^2e^{\frac{Rt}{2}}\right)}{1+c^2e^{\frac{Rt}{2}}}.
$$ 
In this case equation $L_2x=0$ has solution
$$
x(t)=e^{-\displaystyle\int_0^t \dfrac{R\left(1-c^2e^{\frac{Rs}{2}}\right)}{1+c^2e^{\frac{Rs}{2}}}\,ds}\;(>0).
$$

Note also that unlike the case of a differential equation with constant coefficients, in the case of variable coefficients condition
$G_L\subset \mathfrak N$ 
is not necessary for disconjugacy of equation (\ref{ord2.9}) on $\mathbb R$. For example, equation
$$
(\overline L)(t):= x''+tx'+\left(\frac{t^2}{4}+\frac{1}{2}\right)x=0,
$$
having solution $x=e^{-\frac{t^2}{4}}>0$ $(t\in\mathbb R),$
is disconjugate on $\mathbb R$, although $G_{\overline L}\subset \mathfrak O.$ 
The same is true for a more general equation
$$
x''+p(t)x'+\left(\frac{p^2(t)}{4}+\frac{1}{2}p'(t)\right)x=0,\qquad p'(t)>0,\; t\in \mathbb R,
$$
which has solution
$x=e^{-\frac{1}{2}{\displaystyle\int_0^t} p(s)\,ds}>0,\;t \in \mathbb R.$ 

Using the last example and Theorem \ref{thcomp} we derive the following strengthening of Theorem \ref{thnew2}.

\begin{theorem}
\label{thnew3}
If function $p$ is differentiable, $p'(t)\geqslant 0\quad (p'(t)\leqslant 0)$ on $\mathbb R$ and 
$$ 
q(t)\leqslant \frac{p^2(t)}{4}+\frac{1}{2}p'(t)\quad\left(q(t)\leqslant \frac{p^2(t)}{4}-\frac{1}{2}p'(t)\right)\quad (t\in\mathbb R),
$$
then equation $(\ref{ord2.9})$ is disconjugate on $\mathbb R$.
\end{theorem}

\begin{remark}
It is easy to show that a differential equation which is close to disconjugate differential equation, is disconjugate as well.
The latter allows to weaken the assumptions on coefficient $p$, demanding only non-strict monotone increasing (decreasing).
\end{remark}

\textbf{10.} 
Finally, let us consider equation
\begin{equation}\label{ord2.11}
Lx:=  x''+p(t)x'+q(t)x=0 \quad (t\in (a,+\infty))
\end{equation}
with coefficients continuous on $(a,+\infty)$. Substitution $t\to a+t^2$ reduces equation (\ref{ord2.11}) to equation
\begin{equation}\label{ord2.12}
Lx:=  x''+p(a+t^2)x'+q(a+t^2)x=0 \quad (t\in (-\infty,+\infty)).
\end{equation}
Now, disconjugacy of equation (\ref{ord2.12}) on $\mathbb R$ is equivalent to disconjugacy of equation (\ref{ord2.11}) on $(a,+\infty).$ By applying criteria for disconjugacy for equation (\ref{ord2.12}), we get criteria for disconjugacy of equation (\ref{ord2.11}) on $(a,+\infty).$

\end{document}